\documentclass[12pt]{amsart}
\usepackage{fullpage}
\usepackage{amscd, amsaddr}
\usepackage{tikz, tikz-cd, extpfeil}
\usepackage{amsmath, enumitem}
\usepackage{fullpage}

\usepackage{color}
\usepackage{hyperref}

\newtheorem{Theorem}{Theorem}[section]
\newtheorem{Lemma}[Theorem]{Lemma}
\newtheorem{Proposition}[Theorem]{Proposition}

\theoremstyle{definition}
\newtheorem{Definition}[Theorem]{Definition}

\theoremstyle{remark}
\newtheorem{Remark}[Theorem]{Remark}

\newcommand{\C}{\mathbb{C}}

\newcommand{\Z}{\mathbb{Z}}

\newcommand{\g}{\mathfrak{g}}
\newcommand{\crt}{\mathfrak{c}}

\newcommand{\h}{\mathfrak{h}}

\renewcommand{\sl}{\mathfrak{sl}}

\newcommand{\z}{\mathfrak{z}}
\newcommand{\m}{\mathfrak{m}}

\newcommand{\Hyp}{\mathcal{H}}
\newcommand{\Refl}{\mathcal{R}}

\newcommand{\Ad}{\operatorname{Ad}}

\newcommand{\reg}{{\operatorname{reg}}}

\newcommand{\GL}{\operatorname{GL}}

\newcommand{\id}{\operatorname{id}}
\newcommand{\Id}{\operatorname{Id}}

\title[Hyperplane arrangements and Vinberg theta groups]{Hyperplane arrangements and Vinberg's $\theta$-groups}
\author[F. Ambrosio]{Filippo Ambrosio}
\address{FSU Jena, Fakult\"at f\"ur Mathematik und Informatik,\\ Inselplatz 5, 07743 Jena (Germany)}
\email{filippo.ambrosio@uni-jena.de}
\author[A. Santi]{Andrea Santi}
\address{Universit\`a di Roma Tor Vergata, Dipartimento di Matematica,\\ Via della Ricerca Scientifica 1, 00133  Roma (Italy)}
\email{santi@mat.uniroma2.it}
\date{\today}

\begin{document}

\begin{abstract}
Let $\g = \bigoplus_{i \in \Z /m\Z} \g_i$ be a periodically graded semisimple complex Lie algebra.
In this note, we give a uniform proof of the recent result by W. de Graaf and H. V. L\^e that
the hyperplane arrangement determined by the restrictions of the roots of $\g$ to a Cartan subspace $\mathfrak c\subset\g_1$ coincides with the hyperplane arrangement of (complex) reflections of the little Weyl group of $\g = \bigoplus_{i \in \Z /m\Z} \g_i$.
\end{abstract}

\maketitle
\section{Introduction} \label{s_intro}
Let $\g = \bigoplus_{i \in \Z /m\Z} \g_i$ be a periodically graded semisimple complex Lie algebra, and choose any connected semisimple complex algebraic group $G$ with Lie algebra $\g$.
The connected subgroup $G_0$ of $G$ with Lie algebra $\g_0$ is reductive 
and the action of $G_0$ on $\g_1$ naturally induced by the adjoint action of $G$ is called a Vinberg $\theta$-group and denoted for short by $(G_0,\g_1)$. 
Periodically graded semisimple Lie algebras are an important source of representations of complex reductive algebraic groups, originally introduced by Vinberg \cite{Vin, Vin2}, as a non-trivial generalization of the adjoint action of a semisimple algebraic group on its Lie algebra and the isotropy representation of a symmetric space \cite{K, KR}.

Vinberg $\theta$-groups are an active topic of research, see for instance the recent \cite{MR4873513, CES, MR3801422, MR3890218}, and \cite{MR3050830, preprintLR, MR3152016} in relation with Bhargava's advances in the arithmetic theory
of elliptic curves.
They have been investigated over the field of real numbers as well \cite{MR4745882,  MR4407677, dGL, MR3283749}, particularly in connection with their orbit structure and applications to physics  
\cite{MR4644081, MR4452072, MR3608618},
and over fields of characteristic not necessarily zero \cite{Levy, MR3065993, RLYG}. In those three papers, together with \cite{MR2111215}, the relationship between the little Weyl group and the (classical) Weyl group was clarified, gradings of positive rank were classified, and the existence of a Kostant-Weierstrass slice was established. (We refer the reader to \S \ref{s_prelim} for the definitions of the little Weyl group $W$ and its action on a Cartan subspace $\mathfrak c$.)
Finally, Vinberg $\theta$-groups arise naturally also in the context of 
the representation theory of reductive groups over a $p$-adic field $\mathbb{F}$ --  
stable $G_0$-orbits are strictly related to supercuspidal representations of the rational points of $G$ over $\mathbb{F}$ attached to elliptic $Z$-regular elements of the Weyl group \cite{MR3164986} --
and in the context of the character and perverse sheaves on periodically graded Lie algebras \cite{MR4554668, MR3697026,
MR3829497}.

Vinberg $\theta$-groups share many important properties with the adjoint action of a semisimple complex Lie algebra \cite{Vin}: 
\begin{itemize}
\item[(i)] Any element $x\in\g_1$ decomposes uniquely into the sum $x=x_s+x_n$ of two commuting elements $x_s,x_n\in\g_1$ with $x_s$  semisimple and $x_n$ nilpotent;
\item[(ii)] Any two Cartan subspaces of $\g_1$ are conjugate under $G_0$ and any semisimple element of $\g_1$ is contained in a Cartan subspace;
\item[(iii)] The $G_0$-orbit through $x\in\g_1$ is closed if and only if $x$ is semisimple while it is unstable (i.e., its closure contains $0$) if and only if $x$ is nilpotent;
\item[(iv)] There is a Restriction Theorem \`a la Chevalley: the embedding of the Cartan subspace $\mathfrak c$ in $\g_1$ yields an isomorphism of graded algebras $\mathbb C[\g_1]^{G_0}\rightarrow \mathbb C[\mathfrak c]^{W}$;
\item[(v)] $W$ is finite and generated by complex reflections, hence $\mathbb C[\mathfrak c]^{W}$ is a 
polynomial algebra;
\item[(vi)] Two elements $x,y\in\mathfrak c$ are $G_0$-conjugated if and only if they can be mapped one to the other by a transformation from $W$.
\end{itemize}
In particular any element $x\in\g_1$ admits a Jordan decomposition $x=x_s+x_n$ as in (i).
We briefly discuss in \S\ref{s_notation} the relationship of this decomposition with the  more general concept  of a Jordan-Kac-Vinberg decomposition in the sense of \cite[Appendix]{Kac_irs2}, defined for  elements in any representation of a complex reductive algebraic group.
Moreover, like complex semisimple Lie algebras, $\theta$-groups allow for only finitely many nilpotent orbits -- these facts and Chevalley Restriction Theorem (iv) make it possible, at least in principle, to classify $G_0$-orbits 
in $\g_1$.

Nevertheless, the picture is quite richer. For instance, whereas the geometric properties of symmetric spaces are close to those of the adjoint representation, the general case of $\theta$-groups is more interesting from the point of view of reflection groups.
Indeed, the little Weyl group of a symmetric space is itself a Weyl group (for a root system related to the root system of $\g$ in a natural way), but for general $\theta$-groups it is only generated by {\it complex} reflections, finite order linear transformations which fix pointwise an hyperplane of the Cartan subspace. 

In this short note, we focus on another similarity occurring between the case of the adjoint representation and the general periodically graded case.
Fix a Cartan subalgebra $\h \subset \g$ and let $\Phi$ be root system of $\g$ w.r.t. $\h$.
There are several ways to define the Weyl group of $\g$:
one possibility is to define it as the quotient $N_G(\h) / Z_G(\h)$, 
another one as the subgroup of $\GL(\h)$ generated by all the (real) reflections about the hyperplanes in $\Hyp_\Phi=\{\ker \alpha\mid \alpha \in \Phi\}$.
 Although the naive analogue for Vinberg $\theta$-groups of the latter presentation cannot coincide in general with the little Weyl group  (there are multiple complex reflections about the same hyperplane), we illustrate a weaker version of this result that holds for all Vinberg $\theta$-groups. 
Assume that $\g$ admits a $\Z / m \Z$-grading $\g = \bigoplus_{i \in \Z /m\Z} \g_i$ for some positive integer $m$ and fix a Cartan subspace $\crt\subset\g_1$.
Then it is possible to embed $\crt$ in a homogeneous Cartan subalgebra $\h \subset \g$ in such a way that the degree $1$ component of $\h$ is precisely $\crt$.
By considering the set $\Sigma$ of nontrivial restrictions to $\crt$ of the roots in $\Phi$, we obtain the collection $\Hyp_\Sigma=\{\ker \sigma\mid \sigma \in \Sigma\}$ of hyperplanes in $\crt$. 
It then turns out that the following result holds:
\begin{Theorem}
\label{Thm1}
Let $\g=\bigoplus_{i \in \Z /m\Z} \g_i$ be a periodically graded semisimple complex Lie algebra, with Cartan subspace $\crt \subset \g_1$.
Then the hyperplane arrangement  $\Hyp_\Sigma$  induced by restrictions of roots to $\crt$ coincides with the hyperplane arrangement $\Hyp_W$ arising from complex reflections in the little Weyl group $W$.
\end{Theorem}
This result was recently proved by W. de Graaf and H. V. L\^e \cite[Theorem 3.1]{dGL} by means of a case-by-case analysis and a computer-assisted proof.
Our aim is to give a uniform proof of this result, which is based solely on the geometric properties underlying Vinberg $\theta$-groups. This also provides the means for explicit constructions of complex reflections in $W$ associated to a given restricted root $\sigma\in\Sigma$, as explained in \S\ref{lastsection}. In particular, we treat the case of outer diagram automorphisms in a  uniform manner.  
\medskip\par\noindent
{\bf Structure of the paper.} 
In \S\ref{s_notation} we  recall some basic preliminaries on Vinberg $\theta$-groups and 
establish an important auxiliary result, needed in the proof of the main Theorem \ref{Thm1}. The latter is carried out in \S\ref{s_main}, and
we conclude  in \S\ref{lastsection} with further considerations and examples.
\smallskip\par\noindent
{\bf Notations.} 
Throughout the paper we work over the field of complex numbers $\C$, and we fix a positive integer $m$ and a primitive $m$-th root of unity $\omega$.
If $K$ is a complex algebraic group, we denote its identity component by $K^\circ$ and its Lie algebra by $\mathfrak{k}$. We denote the center of $K$ by $Z(K)$ and the center of $\mathfrak{k}$ by $\z(\mathfrak{k})$.
The adjoint action of $K$ on $\mathfrak{k}$ is denoted by  $\Ad \colon K \to \GL(\mathfrak{k}), k \mapsto \Ad_k$, and for any $x \in \mathfrak{k}$ we set  $K x \coloneqq \{\Ad_k (x) \mid k \in K\}$. 
We will write $N_K(X) := \{k \in K \mid \Ad_k(x) \in X \mbox{ for all } x \in X\}$ for the normalizer of any linear subspace $X \subset \mathfrak{k}$ and $Z_K(X) = \{k \in K \mid \Ad_k (x) = x \mbox{ for all } x \in X \}$ for the centralizer of any subset $X \subset \mathfrak{k}$.
They are both algebraic subgroups of $K$, with  corresponding Lie algebras 
$\mathfrak{n}_{\mathfrak k}(X)=\{ y \in \mathfrak{k} \mid [y,x]\in X \mbox{ for all } x \in X\}$ and
$\mathfrak{k}^X = \{ y \in \mathfrak{k} \mid [y,x] = 0 \mbox{ for all } x \in X\}$, respectively.
If $X = \{x\}$ is a singleton, we simply write $K^x:= Z_K(X)$ and $\mathfrak{k}^x := \mathfrak{k}^X$.

\section{Basic notions and preliminary results} \label{s_notation}
\subsection{Vinberg's $\theta$-groups}
A \emph{periodically graded semisimple Lie algebra} is given by a triple $\{\g, \theta, m\}$, where 
$\theta \colon \g \to \g$ is an automorphism of order $m$ of a semisimple
Lie algebra $\g$.
Indeed, the automorphism $\theta$ endows $\g$ with a $\Z / m \Z$-grading, i.e., a direct sum decomposition
\begin{equation} \label{eq_graded}
\g = \bigoplus_{i \in \Z /m\Z} \g_i\quad\text{with}\quad [\g_i , \g_j] \subset \g_{i+j}\quad\text{for all}\quad i, j \in \Z/m\Z\;. 
\end{equation}
In particular $\g_0$ is a Lie subalgebra of $\g$ and any 
$\g_i = \{ x \in \g \mid \theta(x) = \omega^i x\}$ a module for 
$\g_0$.
Conversely, given a $\Z / m \Z$-grading of $\g$ as in \eqref{eq_graded}, we can define an automorphism of $\g$ whose eigenspaces are the homogeneous components of \eqref{eq_graded} (and whose order is a divisor of $m$). 
An element $x \in \g$ is called homogeneous if $x \in \g_i$ for some $i \in \Z / m \Z$. Similarly, a subset $X \subset \g$ is called homogeneous if $X = \bigoplus_{i \in \Z/ m\Z} (X \cap \g_i)$, equivalently, if $X$ is $\theta$-stable.

Let $G$ be any connected semisimple complex algebraic group with Lie algebra $\g$ (for instance, 
the adjoint group) and $G_0$ be the connected subgroup of $G$ with Lie algebra $\g_0$. It is a closed reductive subgroup. The restriction of the adjoint action $\Ad \colon G \to \GL(\mathfrak{\g})$  gives a representation of $G_0$ on $\g_1$ and the pair $(G_0, \g_1)$ is usually referred to as a Vinberg $\theta$-group.

We fix a nondegenerate bilinear form 
\begin{equation} \label{eq_form}
(\cdot, \cdot) \colon \g \times \g \to \C
\end{equation} which is
\begin{enumerate}[label = (\alph*)]
\item[(i)] associative, i.e., $([x,y],z) = (x,[y,z])$ for all $x,y,z \in \g$;
\item[(ii)] $\theta$-invariant, i.e., $(\theta(x), \theta(y)) = (x,y)$ for all $x,y \in \g$.
\end{enumerate} 
The existence of such a form is provided by the Killing form of $\g$. 
For all $i,j\in\Z /m \Z$ with $i + j \neq 0$, we have $(\g_i , \g_j) = 0$, thanks to $\theta$-invariance.
In particular $\g_i \cong (\g_{-i})^*$ as representations of $G_0$ and 
$\dim \g_i = \dim \g_{-i}$, for all $i \in \Z / m \Z$.

For any given $x \in \g_1$, the semisimple  and nilpotent parts of its Jordan decomposition in $\g$ still belong  to $\g_1$, thus inducing a Jordan decomposition on $\g_1$: these are the unique elements $x_s,x_n\in\g_1$ with $x_s$ semisimple, $x_n$ nilpotent, such that $x=x_s+x_n$ and $[x_s,x_n]=0$.

\begin{Remark}
As an interesting aside, one may ask to compare this decomposition with the concept of a Jordan-Kac-Vinberg (for short JKV) decomposition  for the representation of $G_0$ on $\g_1$.
The existence of such a decomposition for any rational representation of a connected complex reductive group is established in \cite[Appendix]{Kac_irs2}. 
Let us briefly recall the definition of a JKV decomposition in our context:
for
$x \in \g_1$, a decomposition
$x = {\mathrm x}_s + {\mathrm x}_n$ is a JKV if  ${\mathrm x}_s, {\mathrm x}_n \in \g_1$, 
the orbit $G_0 {\mathrm x}_s$ is closed, the orbit $G_0^{\mathrm x_s} {\mathrm x}_n$ closes at $0$, and $G_0^x \subset G_0^{\mathrm x_s}$. 
The classical Jordan decomposition of $x \in \g_1$ satisfies the axioms of a JKV decomposition.

For the adjoint representation of $G$ the axioms of a JKV decomposition are equivalent to those of  the classical Jordan decomposition, so that there is a unique JKV decomposition.
In general, this is not the case for Vinberg $\theta$-groups whose order $m\geq 2$, as illustrated in the following example.

Let us consider the symmetric grading $\g=\g_{0}\oplus\g_1$ of $\g=\mathfrak{sl}_{2n}(\mathbb C)$ given by the diagonal $n\times n$ blocks 
$\g_0\cong\mathfrak{s}\big(\mathfrak{gl}_{n}(\mathbb C)\oplus \mathfrak{gl}_{n}(\mathbb C)\big)$ and the off-diagonal $n\times n$ blocks $\g_1\cong M_n(\mathbb C)\oplus M_n(\mathbb C)$. 
The group $G=\mathrm{SL}_{2n}(\mathbb C)$ and the associated Vinberg $\theta$-group is given by 
$$G_0=\left\{g=(g_1,g_2)\mid g_1,g_2\in \mathrm{GL}_n(\mathbb C)\;\text{s.t.}\;\det(g_1)\det(g_2)=1\right\}\cong {\mathrm S }\big(\mathrm{GL}_n(\mathbb C)\times \mathrm{GL}_n(\mathbb C)\big)$$ acting on any $x=(x_+,x_-)\in\g_1$ via 
$g\cdot x:=(g_1 x_+ g_2^{-1},g_2 x_- g_1^{-1})$.
The subalgebra $\g_0$  of $\g$ acts on $\g_1$ by the infinitesimal version of the latter equations, while the bracket between two odd elements is given by $[x,y]=[(x_+,x_-),(y_+,y_-)]=(x_+ y_- -y_+ x_-, x_- y_+ - y_- x_+) \in \g_0$. 

Consider now ${\mathrm x}_s:=(\operatorname{Id}_n,-\operatorname{Id}_n)\in \g_1$.
 The orbit $G_0{\mathrm x}_s$ is closed, since ${\mathrm x}_s$ is semisimple, and the stabilizer $G_0^{{\mathrm x}_s}$ coincides with the diagonal embedding $\operatorname{diag}(\mathrm{SL}^{\pm}_n(\mathbb C))\subset G_0$ of the group of $n\times n$ matrices with determinant $\pm 1$.
 For any non-zero nilpotent matrix $v \in M_n(\mathbb C)$, we set
${\mathrm x}_n:=(v,0)\in\g_1$ and note that the orbit 
$
G_0^{{\mathrm x}_s}{\mathrm x}_n=(\mathrm{SL}^{\pm}_n(\mathbb C)v,0)\cong \mathrm{SL}^{\pm}_n(\mathbb C)v
$
identifies with the usual adjoint orbit of $v$ in $\mathfrak{sl}_n(\mathbb C)$, thus closing at $0$ by classical results.
We finally let $x:={\mathrm x}_s+{\mathrm x}_n$ and claim that this is a JKV decomposition, since $G_0^{x}=\operatorname{diag}(Z_{\mathrm{SL}^{\pm}_n(\mathbb C)}(v))\subset G_0^{{\mathrm x}_s}$. However, it is not a Jordan decomposition, because $[{\mathrm x}_s,{\mathrm x}_n]=(v,-v)\neq 0$ in $\g_0$.

This shows that there exist gradings for which some of their elements $x\in\g_1$ admit more than one JKV decomposition. However, at the present time, the authors were not able to find any such example for which the action of $G_0$ on $\g_1$ is irreducible.
It might thus be an interesting problem to characterize the Vinberg $\theta$-groups for which the JKV decomposition of any  element is unique (thus coinciding with the Jordan decomposition). 
\end{Remark}

\subsection{Cartan subspaces and preliminary results}
\label{s_prelim}

A \emph{Cartan subspace} is an abelian subspace $\crt$ of $\g_1$ consisting of semisimple elements, and which is maximal with these properties. The dimension $r=\dim\mathfrak c$ of any Cartan subspace is called the \emph{rank} of $\{ \g, \theta \}$. Given a Cartan subspace $\crt$, it is always possible to find a homogeneous Cartan subalgebra $\h$ of $\g$ such that $\crt = \h \cap \g_1$, see \cite[\S 3.1]{RLYG}. 
We fix once and for all  $\crt \subset \h \subset \g$ with the above properties and let $\g = \h \oplus \bigoplus_{\alpha \in \Phi} \g_\alpha$ be the root space decomposition of $\g$ w.r.t. $\h$.
Then $\theta$ induces an action on $\Phi$ via $\theta \cdot \alpha \coloneqq \alpha \circ \theta^{-1}$ and root subspaces are permuted via $\g_{\theta\cdot\alpha} = \theta (\g_{\alpha})$.

The \emph{little Weyl group} $W := N_{G_0} (\crt) / Z_{G_0} (\crt)$
is a finite subgroup of $\GL(\crt)$ generated by complex reflections\footnote{In the rest of the paper, we will always talk about reflections omitting the adjective ``complex''.}, i.e., linear transformations $s:\crt\to \crt$  of finite order and such that $\dim \ker(s - \id_\crt) = \dim \crt -1$.
We let $\Refl(W)$ be the subset of $W$ consisting of reflections and $\Hyp_W$ the hyperplane arrangement on $\crt$ consisting of all the hyperplanes $\Pi_s \coloneqq \ker(s - \id_\crt)$ for $s \in \Refl(W)$.
Now, consider the natural inclusion $\rho \colon \crt \to \h$.
We denote by $\Sigma \coloneqq \{ \beta \circ \rho \mid \beta \in \Phi \} \setminus \{0\} \subset \crt^*$ the set of restricted roots and by $\Hyp_{\Sigma}$ the hyperplane arrangement on $\crt$ consisting of all the hyperplanes $\Pi_{\sigma} \coloneqq \ker \sigma$ for $\sigma \in \Sigma$. Equivalently, this is the collection $\{ \ker \beta \cap \crt \mid \beta \in \Phi \} \setminus \{ \crt \}$ of subspaces of $\crt$. 
\\

We believe the following result should be known as a consequence of the full classification of gradings of positive rank and their little Weyl groups \cite{Levy, MR3065993, RLYG}, but we could not locate a simple and direct proof in the literature, so we provide it here.

\begin{Proposition} \label{lem_weylnontriv}
Let $\{ \g, \theta, m \}$ be a  periodically graded semisimple Lie algebra of positive rank.
Then the little Weyl group $W$ of $\{ \g, \theta, m \}$ is nontrivial.
\end{Proposition}

\begin{proof}
Let $\dim \crt \geq 1$, and assume by contradiction that $W$ is trivial. 
By the restriction theorem \`a la Chevalley, the restriction of polynomial functions $\C[\g_1] \to \C[\crt]$ yields an isomorphism of {\it graded} algebras  $\C[\g_1]^{G_0} \to \C[\crt]^W = \C[\crt]$. 
Therefore, there exists a nonzero $G_0$-invariant  $v^* \in \C[\g_1]_1= \g_1^*$.
Via the pairing \eqref{eq_form} we may identify $\g_1^*$ with $\g_{-1}$, so there exists a nonzero $v \in \g_{-1}$ s.t. $G_0 v = \{ v \}$.
Thus $\g_0^v=\g_0$  and by  \cite[(iii) of Corollary  11]{CES}, applied to the grading of $\g$ where $\g_i$ and $\g_{-i}$ exchange their roles for every $i\in \Z / m \Z$, we get that $v\in \z(\mathfrak{\g})$. 
Since $\g$ is semisimple, then $\z(\mathfrak{\g})=0$ and $v=0$, a contradiction.
\end{proof}
\begin{Remark} 
\label{rem:useful}\hfill
\begin{itemize}
	\item[(i)] The reader desiring a self-contained argument in place of \cite[(iii) of Corollary  11]{CES}, may use the following one. Since $[\g_0, v] = 0$, we have $0 = ( [\g_0, v], \g_{1}) = (\g_0 , [v, \g_{1}])$ and 
thus $[\g_{1},v] = 0$, because the restriction of \eqref{eq_form} to $\g_0$ is nondegenerate. Furthermore
$v$ is semisimple, since its orbit is closed.
In particular $\g = \g^v \oplus [\g, v]$, where the sum is orthogonal w.r.t. \eqref{eq_form}, and the restriction of \eqref{eq_form}  to the $\Z / m \Z$-graded Lie algebra $\g^v$ 
is $\theta$-invariant and nondegenerate. Thus $\dim (\g^v \cap \g_{-1}) = \dim (\g^v \cap \g_{1})=\dim\g_1=\dim\g_{-1}$ for all $i \in \Z / m\Z$, and $[\g_{-1}, v] = 0$.
Arguing as above
$([\g_{2}, v], \g_{-1}) = (\g_{2}, [v, \g_{-1}]) = 0$ says $[\g_{2},v] = 0$, and then $[\g_{-2},v]=0$ by looking at the dimensions. We may now iterate the argument to get $[v, \g_i] = 0$ for all $i \in \Z/m\Z$, namely $v \in \z(\g)$.

	\item[(ii)] 	
We	note that the order of $\theta$ does not play any role in this result, which is, in fact, true also in the case $m=1$. However, if $\theta$ is an inner automorphism of order $m\geq 2$, then there exists a lower bound on $W$ that depends on $m$. 
In fact, if $\theta$ is inner, there exists a Cartan subalgebra $\mathfrak h$ of $\g$ that is contained in $\g_0$, see \cite[\S 3.6]{GOV}.
Hence $\theta=\Ad_g$ for some $g\in Z_{G}(\mathfrak h)=H\subset G_0$, where $H$ is the Cartan subgroup corresponding to $\h$, and the class of $g$ in $W=N_{G_0}(\crt)/Z_{G_0}(\crt)$ has order $m$. Thus $W$ includes a subgroup isomorphic to $\Z/ m \Z$.
This bound has already been observed in \cite[Lemma 23]{RLYG}.
\end{itemize}
\end{Remark}

\section{Proof of the main result} \label{s_main}
We come to the proof of  Theorem \ref{Thm1}. 
The claim is trivial if the rank of $\{ \g, \theta, m \}$ is zero (with $\Hyp_W=\Hyp_\Sigma=\varnothing$), so we assume that the rank is positive.
\vskip0.3cm\par\noindent
{\it The inclusion $\Hyp_W \subset \Hyp_\Sigma$.}
\vskip0.3cm\par\noindent
We assume that $x \in \crt$ is such that $x \notin \bigcup_{\sigma \in \Sigma} \Pi_{\sigma}$ and proceed to show that $x \notin \bigcup_{s \in \Refl(W)} \Pi_s$.
We have that 
\begin{equation*}
\g^x = \h \oplus \bigoplus_{\alpha \in \Phi, \alpha(x) = 0} \g_{\alpha}\\
 = \h \oplus \bigoplus_{\alpha \in \Phi, \alpha \circ \rho  = 0 } \g_{\alpha} = \g^{\crt}\;.
\end{equation*}
This implies $G^x = Z_G(\crt)$, since both centralizers are connected due to \cite[Corollary 3.11]{StTor}.
In particular $N_{G_0}(\crt) \cap G^x \subset  Z_{G_0}(\crt)$, whence the stabilizer of $x$ in $W$ is the trivial group, and $x \notin \Pi_s$ for all $s \in \Refl(W)$.

\vskip0.3cm\par\noindent
{\it The inclusion $\Hyp_\Sigma \subset \Hyp_W$.}
\vskip0.3cm\par\noindent
For $\sigma \in \Sigma$, we shall exhibit an element $s\in W$ such that $s \neq \id_\crt $ and $s(x) = x$ for all $x \in \Pi_{\sigma}$, in particular  $s \in \Refl(W)$.
Choose $x$ in the regular locus  of $\Pi_{\sigma}$, namely 
\begin{equation} \label{eq_reglochyp}
x \in \Pi_{\sigma} \setminus \bigcup_{\tau \in \Sigma, \Pi_\tau \neq \Pi_\sigma} \Pi_{\tau}\;.
\end{equation}
We now consider the semisimple Lie algebra $\m := [\g^x, \g^x]$ with the induced 
$\Z / m \Z$-grading and  show that $\{ \m, \theta_{\mid \m}, m \}$ has rank one. 
First, the Levi subalgebra $\g^x$ of $\g$ is homogeneous and
 $\g^x  = \h \oplus \bigoplus_{\alpha \in \Phi_x} \g_{\alpha}$, 
with  $\Phi_x := \{\alpha \in \Phi \mid \alpha(x) = 0 \}\subset \Phi$ a root subsystem. Since $x$ is in the regular locus of $\Pi_{\sigma}$, we have
\begin{equation}
\label{phix}
\Phi_x=\{\alpha \in \Phi \mid \alpha \circ \rho \in \C \sigma \}\supset\{\alpha \in \Phi\mid \ker \alpha\supset \crt\}\;.
\end{equation}
The two summands in the decomposition of $\g^x$ are homogeneous, using that $\Phi_{x}$ is $\theta$-stable. 
The subalgebra $\g^x$ also admits the direct sum decomposition $\g^x = \z(\g^x) \oplus \m$, explicitly
\begin{equation*} 
 \g^x  = \rlap{$\overbrace{\phantom{\z(\g^x)\oplus(\h \cap \mathfrak{g}^x_{ss})}}^{\h}$}
\z(\g^x)
\oplus\underbrace{(\h \cap \m) \oplus \bigoplus_{\alpha \in \Phi_x} \g_{\alpha}}_{\m}\;,
\end{equation*}
where all three summands are homogeneous.
\vskip0.3cm\par\noindent
\emph{Claim I: the equalities $\Pi_\sigma = \z(\g^x) \cap \crt = \z(\g^x)_1$ hold.}
\vskip0.3cm\par\noindent
Since $\z(\g^x)\subset\h$ and $\h_1=\crt$, it is sufficient to show that $\Pi_\sigma=\z(\g^x)\cap\crt$.
Let $q \in \Pi_\sigma$, $y \in \g^x$,
and write $y = y_0 + \sum_{\alpha \in \Phi_x} y_\alpha$ with $y_0 \in \h$, $y_\alpha\in\g_\alpha$.
Then 
\begin{equation*}
[q,y] = \sum_{\alpha \in \Phi_x} \alpha(q) y_\alpha  = 0\;,
\end{equation*} since $\alpha \circ \rho \in \C \sigma$. This proves
the inclusions $\Pi_\sigma\subset\z(\g^x)\cap\crt\subset\crt$, in particular we see that $\dim\Pi_\sigma=\dim \crt - 1 \leq \dim \z(\g^x) \cap \crt \leq \dim \crt$.
Finally, choosing $p \in \crt \setminus \Pi_\sigma$, it is clear that $p\in\g^x$ and $[p , \g^x]=[p, \bigoplus_{\alpha \in \Phi_x} \g_\alpha ] \neq 0$, thus $\crt$ is not contained in $\z(\g^x)$.
\vskip0.3cm\par\noindent
\emph{Claim II: the equality $\dim \crt \cap \m = 1$ hold.}
\vskip0.3cm\par\noindent
Choose $p \in \crt$ s.t. $\crt = \Pi_\sigma \oplus \C p$. By Claim I, we have
$
p \in \g^x_1 = \z(\g^x)_1 \oplus \m_1 = \Pi_\sigma \oplus \m_1
$ and thus we may write $p = z + p_\sigma$ with $z \in \Pi_\sigma$ and $p_\sigma \in  \m_1$.
Because $p, z \in \crt$, then also $p_\sigma \in \crt$.
From $0 \neq \sigma(p) = \sigma(p_\sigma)$ we conclude that $p_\sigma \neq 0$ and
$\crt = \Pi_\sigma \oplus \C p_\sigma$, proving the claim.
\vskip0.3cm\par\noindent
\begin{Remark}
Claim II implies in particular that the order of $\theta_{\mid \m}$ is precisely
$m$ (and not just a divisor of it). 
It also says that $\{ \m, \theta_{\mid \m}, m \}$ has positive rank, and this would be enough to conclude our proof in view of Proposition \ref{lem_weylnontriv}. However, it is not much more difficult to show that the rank is precisely one, as we now do. 
\end{Remark}
\vskip0.3cm\par\noindent
\emph{Claim III: the line $\crt \cap \m$ is a Cartan subspace of $\{ \m, \theta_{\mid \m}\}$.}
\vskip0.2cm\par\noindent
Choose $p_{\sigma} \in \m$ s.t. $\crt=\Pi_\sigma \oplus \C p_{\sigma}$ as in Claim II. 
Clearly $\C p_{\sigma} \subset \m_1$ consists of commuting semisimple elements, 
so it lies in a Cartan subspace $\crt_\m$ for $\m$.
Since $[\m, \Pi_\sigma] = 0$, the subspace $\crt_\m \oplus \Pi_\sigma$ is abelian and consists of semisimple elements.
Hence $\dim ( \crt_\m \oplus \Pi_\sigma) \geq \dim \Pi_\sigma + 1$, which agrees with the rank of $\{ \g, \theta \}$, and $\crt_\m \oplus \Pi_\sigma$ is a Cartan subspace of $\{ \g, \theta \}$ by maximality. In summary $\dim \crt_\m = 1$ and $\crt_\m = \C p_\alpha$, proving Claim III.
%
%
\vskip0.3cm\par\noindent
We complete the proof.
Let $M$ be the closed connected subgroup of $G$ with Lie algebra $\m$, and $M_0$ the closed connected subgroup of $M$ with Lie algebra $\m_0$.
We have the inclusions $M \subset G^x=Z_G(\Pi_\sigma)$ and $M_0 \subset G_0$.
The little Weyl group $W(\m, \theta_{\mid \m})= N_{M_0} (\C p_\sigma) / Z_{M_0} (\C p_\sigma)$ is 
a finite (unitary)
reflection group acting irreducibly on a line, so it is a cyclic group (see, e.g., \cite[Theorem 8.29]{LT}).
It is nontrivial by Proposition \ref{lem_weylnontriv}.
The sought element $s \in \Refl(W)$ can be induced by any choice of nontrivial element in $W(\m, \theta_{\mid \m})$ by means of the natural injection 
\vskip0.2cm\par\noindent
\begin{equation*}
W(\m, \theta_{\mid \m})=N_{M_0} (\C p_\sigma) / Z_{M_0} (\C p_\sigma) \hookrightarrow W=N_{G_0} (\crt) / Z_{G_0} (\crt)\;,
\end{equation*}
\vskip0.2cm\par\noindent
since all $m \in M$ satisfy $\Ad_m{}_{\mid \Pi_\sigma} = \id_{\Pi_\sigma}$. The proof is completed. $\square$
\begin{Remark} 
\label{rk_dk}
Part of the arguments in the proof of Theorem \ref{Thm1} can be deduced also from a construction in the work of Dadok and Kac \cite{DK} on polar representations, a wider class of representations including Vinberg $\theta$-groups.
Such generality exceeds the scope of this note, so we limit ourselves to stating the relevant results only in the setting of graded Lie algebras.

We start by introducing $G_0$-regularity for elements of the Cartan subspace: we define
$$\crt^{G_0-\reg} = \{ y \in \crt \mid \dim G_0y = \max_{x \in \crt} \{\dim G_0x\} \}$$
and say that any $y \in \crt^{G_0-\reg}$ (or its orbit $G_0 y$) is $G_0$-regular. By  \cite[Lemma 2.11]{DK} the singular locus $\crt \setminus \crt^{G_0-\reg}$ is the union of a family of hyperplanes in $\crt$, which coincides with our $\Hyp_\Sigma$.
Indeed, by \cite[Corollary 11 and Example 2]{CES}, an element  $y \in \crt$ is $G_0$-regular if and only if $\dim G y$ is maximal among the $G$-orbits of semisimple elements in $\g_1$.
In other words, $y \in \crt$ is $G_0$-regular if and only if $\g^y = \g^\crt$ and, in turn, if and only if $y \in \crt \setminus \
 \bigcup \{\Pi_\sigma \mid \sigma \in \Sigma \}$.

Dadok and Kac then  produce, for each $\Pi_\sigma \in \Hyp_\Sigma$,  a connected subgroup $G_0^\sigma \subset G_0$ and a vector subspace $\g_1^\sigma \subset \g_1$ such that  $G_0^\sigma$ acts on $\g_1^\sigma$ by restriction as a polar representation,  and show that $W^\sigma := N_{G_0^\sigma} (\crt) / Z_{G_0^\sigma} (\crt)$ is a cyclic subgroup of $W$ consisting of (complex) reflections.
For the reader's convenience, we determine $(G_0^\sigma , \g_1^\sigma)$ in our situation.
The group $G_0^\sigma$ is defined as the connected subgroup of $G_0$ with Lie algebra $\g^{\Pi_\sigma}_0$, namely 
it is $G_0^\sigma = (Z_G(\Pi_\sigma) \cap G_0)^\circ$ (we recall that $Z_G(\Pi_\sigma)$ is always connected but $Z_G(\Pi_\sigma) \cap G_0 $ is not necessarily so), and the vector subspace
$\g_1^\sigma= \crt \oplus [\g_0^{\Pi_\sigma}, \crt] \oplus U$, with $U$ a $\g_0^\crt$-invariant complement to $\crt \oplus [\g_0, \crt]$ in $\g_1$. A direct computation then shows that $U=[\g^{\mathfrak c},\g^{\mathfrak c}]_1$ and $\g_1^\sigma  = \g_1^{\Pi_\sigma}$, whence the action of $G_0^\sigma$ on $\g_1^\sigma$ is again a Vinberg $\theta$-group.

Applying the construction to  Vinberg $\theta$-groups, we get the inclusion $\langle W^\sigma \mid \Pi_\sigma \in \Hyp_{\Sigma} \rangle \subset W$. On the other hand, the group $M_0$ in our proof of Theorem \ref{Thm1} is a subgroup of $G_0^{\sigma}$ and the quotient $W(\m, \theta_{\mid \m})=N_{M_0} (\C p_\sigma) / Z_{M_0} (\C p_\sigma)$ is canonically isomorphic to $W^\sigma$. We may thus invoke our Proposition \ref{lem_weylnontriv} to say that $W^\sigma$ is always nontrivial and, finally, that $\Hyp_\Sigma \subset \Hyp_W$.

We remark that
Dadok and Kac expect that a stronger result holds, namely the identity
 $\langle W^\sigma \mid \Pi_\sigma \in \Hyp_{\Sigma} \rangle = W$, cf. \cite[\S2 Conjecture 2, p. 521]{DK}.
Theorem \ref{Thm1} can be regarded as a weaker version of this conjecture; we plan to study the conjecture in the context of Vinberg $\theta$-groups in another work.
\end{Remark} 

\section{Further considerations and examples} \label{lastsection}

In this final section, we construct an explicit representative in $N_{M_0}(\C p_\sigma)$ lifting the sought reflection for some relevant examples.
The approach can be regarded as a generalization of the idea underlying the classical proof of the analogue of Theorem \ref{Thm1} in the $m = 1$ case, see \cite[Proposition 11.35]{Hall}.

\subsection{Preliminaries}
Let $\{ \g, \theta, m \}$ be a  periodically graded semisimple Lie algebra of positive rank $r=\dim\mathfrak c$ and fix a homogeneous Cartan subalgebra $\h=\bigoplus_{i \in \Z / m\Z}\h_i$ of $\g$ such that $\crt = \h_1$. 
For any root $\alpha \in \Phi$, we choose a root vector $e_\alpha$, and let 
$\mathcal{O}_\alpha:=\left\{\theta^i\cdot\alpha\mid 0\leq i\leq m-1\right\}$ be the orbit of $\alpha$ under the natural action of $\theta$. We denote the cardinality of $\mathcal{O}_\alpha$ by $|\mathcal{O}_\alpha|$. 

We also give the following:

\begin{Definition}\label{def:k}
For any $j\in \Z / m \Z$, we set:
$$
\alpha^{(j)} :=\frac{1}{m}\sum_{i=0}^{m-1}\omega^{-ij}\theta^i\cdot\alpha \in \mathfrak h^*; \qquad \mbox{and} \qquad
e_{\alpha}^{(j)} :=\frac{|\mathcal{O}_\alpha|}{m}\sum_{i = 0}^{m-1} \omega^{-ij}\theta^i(e_{\alpha}) \in \g_j. $$
\end{Definition}


Note that $\alpha^{(j)}$ and $e_{\alpha}^{(j)}$ depend only on $\mathcal{O}_\alpha$, up to non-zero scalars, and that
$\alpha=\sum_{i=0}^{m-1}\alpha^{(j)}$, with
 $\alpha^{(j)}|_{\h_i}=0$ for all $i \neq -j$.
For any $\sigma \in \Sigma$, we discussed in \S\ref{s_main} the reduction process from $\g$ to 
the Levi subalgebra $\g^{\Pi_\sigma}=\g^x  = \h \oplus \bigoplus_{\alpha \in \Phi_x} \g_{\alpha}$ (or rather 
its semisimple part $\mathfrak m=[\g^{\Pi_\sigma},\g^{\Pi_\sigma}]$),
where $x\in\mathfrak c$ and $\Phi_x\subset\Phi$ are as in \eqref{eq_reglochyp}-\eqref{phix}.
For any $\alpha \in \Phi_x$ with $\alpha^{(-1)}=\alpha\circ\rho\neq 0$ we have $|\mathcal{O}_\alpha|=m$, all $e_{\alpha}^{(j)}$ belong to $\mathfrak m$ and are non-zero, and $[e_{\alpha}^{(0)}, p_\sigma]=-\alpha(p_\sigma)e_{\alpha}^{(1)}\neq 0$.

From now on, we normalize \eqref{eq_form} in such a way that the long roots have squared length $2$ and tacitly identify $\h$ with $\h^*$ using \eqref{eq_form}. In particular, any $\alpha^{(j)}\in\mathfrak h_j$.

We will freely use the following result in \S\ref{ss_diagramAeven}-\S\ref{ss_section4} without any explicit mention.
\begin{Lemma}
The Cartan subalgebra $\mathfrak h\cap\mathfrak m$ of $\mathfrak m$ is $\theta$-stable and its degree $j$ component $\mathfrak h_j\cap\mathfrak m$ is generated by the $\alpha^{(j)}$'s running with $\alpha\in\Phi_x$.
\end{Lemma}
\begin{proof}
It follows from the fact that $\mathfrak h\cap\mathfrak m$ is generated by the vectors $[e_\alpha,e_{-\alpha}]$, 
$\alpha\in\Phi_x$.
\end{proof}

We will now construct a representative lifting the sought reflection  by exponentiating certain linear combinations of the $e_{\alpha}^{(0)}$ with $\alpha\in\Phi_x$, in the cases of $\g$ simple of type $A$ with $\theta$ inner and $\g$ simple admitting a diagram automorphism $\theta$.
In \S \ref{ss_diagramAeven} and \S \ref{ss_innerA}, the root vectors $e_\alpha$ are normalized as  matrix units
and we follow the standard expressions and numbering of the simple roots for simple Lie algebras of type $A$.

\subsection{Diagram automorphism of $\mathfrak{sl}_{2n+1}(\C)$}
\label{ss_diagramAeven}
This is an opportunity to revisit the $m=1$ case  under the lens of periodically graded Lie algebras. Let  $\{\g, \theta, m\}= \{\mathfrak{sl}_{2n+1}(\C), \theta, 2\}$ with  $\theta$ the diagram automorphism of $\mathfrak{sl}_{2n+1}(\C)$, so $\g_0\cong \mathfrak{so}_{2n+1}(\C)$ acts on the the traceless second symmetric power $\g_1\cong \odot^2_0\C^{2n+1}$ irreducibly. 
We may thus conjugate $\theta$ so that it coincides with the Chevalley involution
of $\g=\mathfrak{sl}_{2n+1}(\C)$.

The standard Cartan subalgebra of $\g$ given by the diagonal traceless matrices is contained in $\g_1$, so it is a Cartan subspace $\mathfrak c$ of $\{\g, \theta, m\}$, and any $\alpha=\epsilon_k-\epsilon_\ell\in\Phi$ yields $\sigma=\alpha^{(-1)}=\alpha$. Now
\begin{align*}
\mathcal{O}_\alpha&=\left\{\alpha,-\alpha\right\}\;,\quad|\mathcal{O}_\alpha|=2\;,\\
e_{\alpha}^{(j)}&=e_\alpha+(-1)^j\theta(e_\alpha)\;,
\end{align*}
and $\g^{\Pi_\sigma}=\mathfrak c\oplus\g_\alpha\oplus\g_{-\alpha}$. It then follows that $\m=\C\alpha\oplus\g_\alpha\oplus\g_{-\alpha}\cong\mathfrak{sl}_2(\C)$ and $\mathfrak c\cap\mathfrak m=\C\alpha$,
and a simple computation in $\mathfrak{sl}_2(\mathbb C)$ 
yields
\begin{equation*}
[e_{\alpha}^{(0)},\alpha^{(1)}]=-2e_{\alpha}^{(1)}\;,\quad [e_{\alpha}^{(0)},e_{\alpha}^{(1)}]=2\alpha^{(1)}\;.
\end{equation*}
On the plane $\C\alpha^{(1)}\oplus e_{\alpha}^{(1)}\subset\mathfrak m_1$, the element $J:=\frac{1}{2}e_\alpha^{(0)}\in\mathfrak m_0$ acts as a complex structure and its exponential $\exp(\pi J) \in M_0$ as $-\id$. 
This is the sought lift of the reflection.
\subsection{Inner automorphisms of $\sl_{n+1}(\C)$}
\label{ss_innerA}
By the results of \cite[\S 7-\S8]{Vin} and \cite[\S 8]{MR1104219}, any inner automorphism $\theta: \g \to \g$ of $\g=\sl_{n+1}(\C)$ of positive rank is described, up to conjugation, by a crossed Kac diagram.
This is the affine Dynkin diagram of $\g$, where each node has an additional label $1$, if the node is crossed, or $0$, otherwise.
The order $m$ of $\theta$ is the number of crossed nodes, so the Kac diagrams we are considering have 
always at least one cross and we may assume w.l.o.g. that this is placed on the lowest root $\alpha_0$.
In summary, we have an ordered partition $(k_1,k_2,\ldots,k_N)$ of $n+1$, where $k_1$ is the number of consecutive crossed nodes counting  
from $\alpha_0$, $k_2$ is the number of consecutive uncrossed nodes counting from the simple root $\alpha_{k_1}$, etc. 
Thanks to \cite[Proposition 17]{Vin},
$[\g_0,\g_0]\cong\mathfrak{sl}_{k_2+1}(\C)\oplus \mathfrak{sl}_{k_4+1}(\C)\oplus\cdots$ and $\g_1$ is the sum of $m$ irreducible representations of $\g_0$ (one for each crossed node among the following representations
$\mathbb C$, $\mathbb C^{k_{2j}+1}$, $(\mathbb C^{k_{2j}+1})^*$, or $\mathbb C^{k_{2j}+1}\otimes (\mathbb C^{k_{2(j+1)}+1})^*$).

The associated automorphism is  $\theta=\Ad_{\exp( 2 \pi i x)}$, where $x$ is the element of the standard Cartan subalgebra $\mathfrak h$ of $\g$ with coordinates:
$$ \alpha_k(x) = \begin{cases} 1/m & \mbox{ if the node } k \mbox{ is crossed}\\
0 & \mbox{ otherwise} \end{cases}$$ w.r.t. simple roots $\alpha_k$, $1\leq k\leq n$. 
We now trade the automorphism with an equivalent one.
First, up to rescaling $x$ by a non-zero multiplicative factor, we may assume that
\begin{align*}
\exp(2\pi ix) &=\operatorname{diag}(\underbrace{1,\ldots,\omega^{k_1-1}}_{k_1 \text{ elements}},\underbrace{\omega^{k_1-1},\ldots,\omega^{k_1-1}}_{k_2 \text{ elements}},\underbrace{\omega^{k_1},\ldots,\omega^{k_1+k_3-1}}_{k_3 \text{ elements}},\ldots,\omega^{m-1}).
\end{align*}
It can be shown that the rank $r$ is equal to the minimum of the eigenvalue multiplicities thanks to \cite[\S 7]{Vin}.
 This then suggests to perform an additional conjugation, varying $x$ in $\g$ in such a way that $\theta$ is  determined by the block diagonal matrix 
\begin{align}
\label{eq:finaltheta}
\exp(2\pi ix)=\operatorname{diag}(\underbrace{P,\ldots,P}_{r \text{ times}},\omega^{k_1-1}\operatorname{Id}_{k_2+1-r},\omega^{k_1+k_3-1}\operatorname{Id}_{k_4+1-r},\ldots)\;,
\end{align}
where $P= \left( \begin{smallmatrix} 
0 & 1 & 0 & \cdots & 0 \\
0 & 0 & 1 & \cdots & 0 \\
\vdots & \vdots & \vdots & \vdots & \vdots \\
0 & 0 & 0 & 0 & 1 \\
1 & 0 & 0 & 0 & 0 
\end{smallmatrix} \right)$ is the $m\times m$ cyclic permutation matrix.
\vskip0.2cm\par
From now on we consider $\theta=\Ad_{\exp(2\pi ix)}$ as per \eqref{eq:finaltheta} and 
split the standard representation of $\g$ accordingly into $\mathbb C^{n+1}=\mathbb C^{rm}\oplus\mathbb C^{n+1-rm}$.
The Cartan subalgebra $\mathfrak h$ of $\g$ similarly decomposes into $\mathfrak h=\mathfrak t\oplus\mathfrak h'\oplus\mathbb C\mathcal I$, with $\mathfrak t$ (resp. $\mathfrak h'$) the standard Cartan subalgebra of $\mathfrak{sl}_{rm}(\mathbb C)$ (resp. $\mathfrak{sl}_{n+1-rm}(\mathbb C)$), and 
$$
\mathcal
I := \left( \begin{smallmatrix} (n+1 -rm) \Id_{rm} & 0 \\ 0 & -rm \Id_{n+1-rm} \end{smallmatrix} \right)\;.
$$
  We also set 
\begin{equation} 
\label{eq_s}
\begin{aligned}
\mathbf{s}&:=\operatorname{diag}(1,\omega,\omega^2,\ldots,\omega^{m-1})\;,\\
\mathbf{s}_\ell&:=\operatorname{diag}(\underbrace{0_m,\ldots,0_m}_{\ell-1 \text{ times}},\mathbf{s},\underbrace{0_m,\ldots,0_m}_{r-\ell \text{ times}},0_{n+1-rm})\;,
\end{aligned}
\end{equation}
for any $1\leq \ell\leq r$. The latter is the block diagonal matrix with a block $\mathbf{s}$ in position $(\ell, \ell)$ and zero blocks elsewhere.

\begin{Lemma}
The Cartan subalgebra $\mathfrak h$ of $\g$ is homogeneous.
More precisely $\mathfrak h'\oplus\mathbb C\mathcal I\subset \mathfrak h_0$ and $\mathfrak t=\bigoplus_{i \in \Z / m\Z} \mathfrak t_i$ with components 
$$\mathfrak t_{i}=
\begin{cases}
\mathbb C\mathbf{s}_1^{i}\oplus\cdots\mathbb\oplus\,\mathbb C\mathbf{s}_r^i\quad& \text{if}\;1\leq i\leq m-1;\\
\\
\mathbb C\mathbf{s}_1^{0}\oplus\cdots\mathbb\oplus\,\mathbb C\mathbf{s}_r^0\cap\;\mathfrak{sl}_{rm}(\mathbb C)\quad& \text{if}\;i=0.
\end{cases} 
$$
In particular, a Cartan subspace of $\{\g, \theta, m\}$ is given by $\mathfrak c=\mathfrak t_1=\mathbb C\mathbf{s}_1\oplus\cdots\mathbb\oplus\,\mathbb C\mathbf{s}_r$. 
\end{Lemma}

The proof is straightforward and we omit it. If $\beta=\epsilon_i-\epsilon_j\in\Phi$ is any positive root of $\mathfrak g$ (i.e., $1\leq i<j\leq n+1$), then $\sigma=\beta\circ\rho=0$ if and only if $i,j>rm$, whereas:
\begin{itemize}
	\item[(i)] If $(\ell-1) m+1\leq i<j\leq \ell m $ for some $1\leq \ell\leq r$, or instead
	$(\ell-1) m+1\leq i\leq \ell m $ and $j>rm$ for some $1\leq \ell\leq r$,
	then  
%
$\Pi_\sigma=\langle \mathbf{s}_1, \dots, \mathbf{s}_{\ell - 1}, \mathbf{s}_{\ell + 1}, \dots, \mathbf{s}_r \rangle$;
	\item[(ii)] If $(\textit{k}-1) m+1\leq i\leq \textit{k}\, m $ and $(\ell-1) m+1\leq j\leq \ell m $ for some
	$1\leq \textit{k}<\ell\leq r$, then
$\Pi_\sigma= \langle \omega^{-i}\mathbf{s}_k +\omega^{-j}\mathbf{s}_{\ell}, \mathbf{s}_1, \dots,  \mathbf{s}_{k-1} , \mathbf{s}_{k +1} , \dots, 
 \mathbf{s}_{\ell-1} , \mathbf{s}_{\ell+1}, \dots , \mathbf{s}_r \rangle$.
\end{itemize}
We are now ready to state and prove the following.
\begin{Theorem}
\label{thm:innersl}
Let $\sigma\in\Sigma$ with the corresponding hyperplane $\Pi_\sigma$ in $\mathfrak c$. Then there are two possibilities:

\begin{enumerate}
	\item $\Pi_\sigma=\langle \mathbf{s}_1, \dots, \mathbf{s}_{\ell - 1}, \mathbf{s}_{\ell + 1}, \dots, \mathbf{s}_r \rangle$ for some $1\leq \ell\leq r$ and the lift of a complex reflection $s\in W$ with $\Pi_s=\Pi_\sigma$ is provided by 
$$
\exp\big(\sum_{k=1}^{m-1}a_k e_{\alpha_{(\ell-1) m+1}+\cdots+\alpha_{(\ell-1) m+k}}^{(0)}\big)\;,
$$
\vskip0.1cm\par\noindent
where the coefficients $a_1,\ldots,a_{m-1}\in\mathbb C$ are  determined in Proposition \ref{prop:uniquesolution} later on;
\item $\Pi_\sigma= \langle \omega^{-i}\mathbf{s}_k +\omega^{-j}\mathbf{s}_{\ell}, \mathbf{s}_1, \dots,  \mathbf{s}_{k-1} , \mathbf{s}_{k +1} , \dots, 
 \mathbf{s}_{\ell-1} , \mathbf{s}_{\ell+1}, \dots , \mathbf{s}_r \rangle$ for some $1\leq \textit{k}<\ell\leq r$,
$(\textit{k}-1) m+1\leq i\leq \textit{k}\, m $, $(\ell-1) m+1\leq j\leq \ell m $, 
and the lift of a complex reflection $s\in W$ with $\Pi_s=\Pi_\sigma$ is provided by
$$
\exp\big(\frac{\pi}{2} (e_\alpha^{(0)}  - e_{-\alpha}^{(0)})\big)\;\;\;\;\;\;\;\;\;\;\;\;\;\;\;\;\;\;\;\;\;\;
$$
where $\alpha = \epsilon_{i}-\epsilon_j$.
\end{enumerate}
\end{Theorem}
\noindent
We first focus on the proof of Theorem \ref{thm:innersl} for case (1).
 Using the above discussions and \eqref{phix}, we readily get 

\begin{equation*}
\begin{aligned}
\Phi_x=&\left\{\epsilon_p-\epsilon_q\mid
(\ell-1) m+1\leq p,q\leq \ell m, p\neq q\right\}\bigcup \\
&\left\{\pm(\epsilon_p-\epsilon_q)\mid
(\ell-1) m+1\leq p\leq \ell m, q>rm\right\}\bigcup \\
&\left\{\epsilon_p-\epsilon_q\mid
p,q>rm, p\neq q\right\}
\end{aligned}
\end{equation*}
so the semisimple Lie algebra $\mathfrak m=[\g^{\Pi_\sigma}, \g^{\Pi_\sigma}]$ is isomorphic to $\mathfrak{sl}_{n+1-(r-1)m}(\mathbb C)$. Inside $\mathfrak m$, we can focus on the Lie subalgebra generated by the root spaces that correspond to the first contribution in $\Phi_x$: it is a $\theta$-stable subalgebra isomorphic to $\mathfrak{sl}_m(\mathbb C)$,  and it includes the Cartan subspace $\mathfrak c\cap \mathfrak m=\mathbb C \mathbf{s}_\ell$ of $\mathfrak m$.

With some temporary abuse of notation, we will thus work with the periodically graded semisimple Lie algebra $\{\mathfrak s:=\mathfrak{sl}_m(\mathbb C),\theta:=\Ad_P,m\}$, the standard Cartan subalgebra $\mathfrak h$ with basis 
$\left\{\mathbf{s}, \mathbf{s}^2, \ldots, \mathbf{s}^{m-1}\right\}$, Cartan subspace $	\mathfrak{c}:=\mathbb C \mathbf{s}$, simple roots
$\alpha_k=\epsilon_k-\epsilon_{k+1}$, $1\leq k\leq m-1$.
\begin{Lemma}
\label{lemma:usefulinner}
\hfill\vskip0.2cm\par\noindent
\begin{itemize}
	\item[(i)] $\mathfrak s_0$ has basis $P, P^2, \ldots, P^{m-1}$;
	\item[(ii)] $\theta$-orbits of roots are parametrized by the root height modulo $m$: if $\alpha=\alpha_1+\cdots+\alpha_k$ has $\operatorname{ht}(\alpha)=k\geq 1$, then
\begin{align*}
\;\;\;\;\;\;\;\;\mathcal{O}_\alpha&=\left\{\alpha,\theta^{-1}\cdot\alpha=\alpha_2+\cdots+\alpha_{k+1},\ldots,\theta\cdot\alpha=\alpha_0+\cdots+\alpha_{k-1}\right\}\;,\\
|\mathcal{O}_\alpha|&=m\;,\quad e_{\alpha}^{(0)}=P^k\;.
\end{align*}
\end{itemize}
\end{Lemma}
\begin{proof}
First $\mathfrak s_0=\mathfrak s^P$ contains all the powers $P, P^2, \ldots, P^{m-1}$ of $P$. 
Assertion (i) follows then from the fact that $P$ is a regular semisimple element of $\mathfrak s$ and the above powers are linearly independent. Claim (ii) follows directly from the equations $\theta\cdot\alpha_0=\alpha_{m-1}$, 
$\theta\cdot \alpha_{k}=\alpha_{k-1}$  for all $1\leq k\leq m-1$,
$\theta(e_{\alpha})=e_{\theta\cdot\alpha}$ for all roots $\alpha$.
\end{proof}
In view of (ii) of Remark \ref{rem:useful} and of Lemma \ref{lemma:usefulinner}, we seek coefficients $a_1,\ldots,a_{m-1}\in\mathbb C$ such that
the adjoint action of the exponential of the element $
\sum_{k=1}^{m-1}a_k P^k\in\mathfrak s_0$ is $\theta=\Ad_P$.
 In other words, we look for solutions of the equation 
$
\exp ( \sum_{k=1}^{m-1}a_k P^k ) = \varepsilon P
$,
where an additional constant $\varepsilon$ satisfying $\varepsilon^m=(-1)^{m+1}$ has been introduced to guarantee the unit determinant.
 Now $P$ and $\mathbf{s}$ are semisimple matrices with the same eigenvalues and therefore conjugated, so the equation 
can be traded with $\exp({\sum_{k=1}^{m-1}a_k \mathbf{s}^k})=(-1)^{m+1}\mathbf{s}$, which is easier to study and reads 
\begin{equation}
\label{eq:bigsystem}
\begin{aligned}
a_1+a_2+\cdots+a_{m-1}&=\log\big(\varepsilon\big)\\
a_1\omega+a_2\omega^2+\cdots+a_{m-1}\omega^{m-1}&=\log\big(\varepsilon\omega\big)\\
a_1\omega^2+a_2\omega^4+\cdots+a_{m-1}\omega^{2(m-1)}&=\log\big(\varepsilon\omega^2\big)\\
\vdots\;\;\;\;\;\;\;\;\;\;\;\;\;\;\;\;\;\;\;\;\;\;\;\;\;\;\;\;\;\;&\;\;\;\;\;\vdots\\
\vdots\;\;\;\;\;\;\;\;\;\;\;\;\;\;\;\;\;\;\;\;\;\;\;\;\;\;\;\;\;\;&\;\;\;\;\;\vdots\\
a_1\omega^{m-1}+a_2\omega^{2(m-1)}+\cdots+a_{m-1}\omega^{(m-1)(m-1)}&=\log\big(\varepsilon\omega^{m-1}\big)
\end{aligned}
\end{equation}
(Given $z\in\mathbb C$, the symbol $\log(z)$ denotes any complex number such that $\exp(\log(z))=z$.)
This system appears to be overdetermined, since it has $m$ equations in $(m-1)$ unknowns, so it is convenient to  slightly reformulate it and single out the additional constraint explicitly.
\begin{Proposition}
\label{prop:uniquesolution}
The system \eqref{eq:bigsystem} is equivalent to the square system $A\vec{a}=\log(\vec{\omega})$, with
\vskip0.2cm\par\noindent
\begin{multline*}
A={\scriptsize\begin{pmatrix}
1 & 1 & 1 & \cdots & 1 \\
1& \omega & \omega^2 & \cdots & \omega^{m-1} \\
1 & \omega^2 & (\omega^2)^2 & \cdots & (\omega^{m-1})^{2}\\
\vdots & \vdots & \vdots &  & \vdots \\
1 & \omega^{m-1} & (\omega^2)^{m-1} & \cdots & (\omega^{m-1})^{m-1} 
\end{pmatrix}}\;,\;\;
\vec{a}={\scriptsize
\begin{pmatrix}
a_0\\
a_1\\
\vdots\\
a_{m-2}\\
a_{m-1}
\end{pmatrix}}\;,\;\;
\log(\vec{\omega})={\scriptsize
\begin{pmatrix}
\log(\varepsilon)\\
\log(\varepsilon\omega)\\
\vdots\\
\log(\varepsilon\omega^{m-2})\\
\log(\varepsilon\omega^{m-1})
\end{pmatrix}}\;,
\end{multline*}
\vskip0.2cm\par\noindent
under the additional constraint $a_0=0$. In particular, it has a unique solution $\vec{a}=\frac{1}{m}\overline A \log(\vec{\omega})$.
\begin{proof}
The matrix $A$ is readily seen to be invertible with inverse $A^{-1}=\frac{1}{m}\overline A$ by the classical formulae for determinant and inverse of a square matrix of Vandermonde type. The solution to \eqref{eq:bigsystem} thus exists and is unique, provided the first entry of $\overline A \log(\vec{\omega})$ vanishes, namely
\begin{multline*}
0=\log(\varepsilon)+
\log(\varepsilon\omega)+\cdots+
\log(\varepsilon\omega^{m-2})+
\log(\varepsilon\omega^{m-1})\Longleftrightarrow\\
1=\varepsilon^m 1\cdot\omega\cdot\cdots\cdot\omega^{m-2}\cdot\omega^{m-1}\;,
\end{multline*}
which is true because the product of all roots of unity of order $m$ is equal to $(-1)^{m+1}$.
\end{proof}
\end{Proposition}
The proof of Theorem \ref{thm:innersl} for the case (1) is completed and we now turn to the case (2). Using again \eqref{phix}, we see

\begin{equation*}
\begin{aligned}
\Phi_x=&\left\{\pm(\epsilon_p-\epsilon_{q(p)})\mid
(\textit{k}-1) m+1\leq p\leq \textit{k}\, m\right\}\bigcup \\
&\left\{\epsilon_p-\epsilon_q\mid
p,q>rm, p\neq q\right\}\;,
\end{aligned}
\end{equation*}
where
	$1\leq \textit{k}<\ell\leq r$ and $q(p)$ is the positive integer uniquely determined by $p$ by means of the equations
	 $(\ell-1) m+1\leq q(p)\leq \ell m$, $\overline{q(p)}=\overline{p+j-i}$ in $\Z /m \Z$.
Then $\mathfrak m\cong m\,\mathfrak{sl}_{2}(\mathbb C)\oplus \mathfrak{sl}_{n+1-mr}(\mathbb C)$ with Cartan subspace $\mathfrak c\cap \mathfrak m=\mathbb C (\omega^{-i}\mathbf{s}_{\textit{k}}-\omega^{-j}\mathbf{s}_{\ell})$, and 
its ideal $\mathfrak s:=m\,\mathfrak{sl}_{2}(\mathbb C)$, which is generated by the root spaces corresponding to the first contribution in $\Phi_x$, is 
	$\theta$-stable and includes $\mathfrak c\cap \mathfrak m$.
With a little abuse of notation, we thus consider the
	periodically graded semisimple Lie algebra $\{\mathfrak s:=m\,\mathfrak{sl}_{2}(\mathbb C),\theta,m\}$.
\begin{Lemma}
\label{lemma:usefulouter}
The automorphism $\theta:\mathfrak s\to\mathfrak s$ cyclically permutes the simple ideals of $\mathfrak s$, in particular it is an outer automorphism of $\mathfrak s$. 
It follows that
\begin{itemize}
	\item[(i)] there are  two $\theta$-orbits 	
	 of roots of $\mathfrak s$, i.e., the orbit of all positive roots 
	$$\mathcal O_+:=\left\{\epsilon_p-\epsilon_{q(p)}\mid (\textit{k}-1) m+1\leq p\leq \textit{k}\, m \right\}$$ and the orbit of all negative roots $\mathcal O_-=-\mathcal O_+$, and both orbits have cardinality $m$;
	\item[(ii)] if $\imath:\mathfrak{sl}_2(\mathbb C)\to\mathfrak s$ is the natural diagonal embedding, then
the element $e_{\alpha}^{(0)}\in\mathfrak{s}_0$ associated to the $\theta$-orbit $\mathcal O_\alpha$ of a root $\alpha$ is given by
$$
e_{+}^{(0)}=\imath \begin{pmatrix}0 & 1 \\ 0 & 0\end{pmatrix}\;\;\text{if}\;\mathcal O_\alpha=\mathcal O_+\;,\quad
e_{-}^{(0)}=\imath \begin{pmatrix}0 & 0 \\ 1 & 0\end{pmatrix}\;\;\text{if}\;\,\mathcal O_\alpha=\mathcal O_-\;.
$$
\end{itemize}
\end{Lemma}
\begin{proof}
It follows immediately from the equations $\theta\cdot(\epsilon_p-\epsilon_{q(p)})=\epsilon_{p-1}-\epsilon_{q(p)-1}=\epsilon_{p-1}-\epsilon_{q(p-1)}$ (if $p=(\textit{k}-1) m+1$ then $p-1$ has to be understood as $\textit{k}\,m$ and similarly if $q(p)=(\ell-1) m+1$) 
and 
$\theta(e_{\alpha})=e_{\theta\cdot\alpha}$ for all roots $\alpha$.
\end{proof}
The triple $\{\mathfrak s:=m\,\mathfrak{sl}_{2}(\mathbb C),\theta,m\}$ fits thus into the discussion carried out in \cite[\S1.2, \S2.2]{Vin}, with $\{\mathfrak{sl}_2(\mathbb C),\id,1\}$ as the associated simple component. It follows (see, e.g., \cite[pag. 479]{Vin}) that the little Weyl group of the triple, considered as a linear group, is isomorphic to the Weyl group of the simple component. 
The latter is $\mathbb Z_2$ and the sought lift of the reflection is thus $\exp( \pi J )$ with $J:=\frac{1}{2}\big(e_{+}^{(0)}-e_{-}^{(0)}\big)\in\mathfrak s_0$, arguing as in \S \ref{ss_diagramAeven}. 
This concludes the discussion of case (2) and the proof of Theorem \ref{thm:innersl}.

\subsection{Diagram automorphisms of $\mathfrak{so}_{2n+2}(\C)$, $\mathfrak{sl}_{2n}(\C)$, $E_6$ ($m=2$) and $\mathfrak{so}_{8}(\C)$ ($m=3$)}
\label{ss_section4}
We consider  periodically graded Lie algebras $\{\g, \theta, m\}$ arising from diagram automorphisms $\theta$ of simple Lie algebras $\g$, except for the case $\g=\mathfrak{sl}_{2n+1}(\C)$ already examined in \S\ref{ss_diagramAeven}. 
We recollect them in the following Table \ref{tablediagram}, which also includes the Vinberg $\theta$-group $(G_0,\g_1)$ and the rank $r=\dim\mathfrak c$ (see \cite[Summary Table]{MR1309681}).
\begin{table}[h]
\[
\begin{array}{|c|c|c|c|c} \hline
\g & m & (G_0,\g_1) & r \\ \hline\hline
 \mathfrak{so}_{2n+2}(\C)& 2  & (\mathrm{SO}_{2n+1}(\mathbb C), \C^{2n+1})& 1\\ \hline

 \mathfrak{sl}_{2n}(\C)&2  & (\mathrm{Sp}_{2n}(\mathbb C),\Lambda^2_0\C^{2n}) & n-1\\ \hline

  E_6& 2 & (F_4, \C^{26})  & 2\\ \hline
 \mathfrak{so}_{8}(\C)& 3 & (G_2,\C^7) &1\\ \hline
 \end{array}
 \]
 \caption{}
 \label{tablediagram}
 \end{table}
\par\noindent
We fix a Cartan subalgebra $\h$ of $\g$ and a system of simple roots $\left\{\alpha_1,\ldots,\alpha_N\right\}$ as in \cite[\S 6.7]{MR1104219} and use the uniform description of the simply-laced simple Lie algebras $\g$ and their diagram automorphisms $\theta$ given in \cite[\S 7.8-7.9]{MR1104219}.
 In particular, the subalgebra $\mathfrak h$ is $\theta$-stable, the root vectors $e_\alpha$ satisfy the Chevalley-type relations \cite[7.8.5]{MR1104219}, the automorphism $\theta$ permutes the simple roots of $\g$, and $\theta(e_\alpha):=e_{\theta\cdot\alpha}$ for all $\alpha\in\Phi$.
  Finally,
$\alpha^{(j)}$ and $e_{\alpha}^{(j)}$ in Definition \ref{def:k} are in agreement with the definitions
given in \cite[\S7.9]{MR1104219}, with the  exception that our $\alpha^{(j)}$ has an additional overall factor.
We also recall that we already agreed to tacitly identify $\h$ with $\h^*$ using  \eqref{eq_form}.
\begin{Lemma} The graded component $\mathfrak h_1$ of $\h=\bigoplus_{i \in \Z / m\Z} \h_i$ has dimension equal to the number of $\theta$-orbits $\mathcal O_\alpha$ of simple roots $\alpha$ of cardinality $|\mathcal O_\alpha|>1$.
It is a Cartan subspace $\mathfrak c$ of $\{\g, \theta, m\}$, with basis given in the following Table  \ref{tablebasis}.
\begin{table}[h]
\[
\begin{array}{|c|c|c|c} \hline
\g & m & {\text Basis\; of}\;\mathfrak c\;\; \\ \hline\hline
 \mathfrak{so}_{2n+2}(\C)& 2 &\alpha_{n+1}^{(1)}=\epsilon_{n+1}  \\ \hline

 \mathfrak{sl}_{2n}(\C)& 2 &\alpha_k^{(1)}=\frac{1}{2}\big(\epsilon_k-\epsilon_{k+1}-\epsilon_{2n-k}+\epsilon_{2n-k+1}\big)\;\;\text{for}\;\;1\leq k\leq n-1   \\ \hline

  E_6& 2 &\alpha_1^{(1)}=\frac{1}{2}\big(\epsilon_1-\epsilon_{2}-\epsilon_{5}+\epsilon_{6}\big),\;\alpha_2^{(1)}=\frac{1}{2}\big(\epsilon_2-\epsilon_{3}-\epsilon_{4}+\epsilon_{5}\big) \\ \hline
\mathfrak{so}_{8}(\C)& 3 &\alpha_1^{(1)}=\frac{1}{3}\big(\epsilon_1-\epsilon_2-\epsilon_3+(\omega-\omega^{2})\epsilon_4\big) \\ \hline
 \end{array}
 \]
 \caption{}
 \label{tablebasis}
 \end{table}
\end{Lemma}
\begin{proof}
 The component $\mathfrak h_1$ is generated by $\alpha_1^{(1)},\ldots,\alpha_N^{(1)}$, but 
\begin{itemize}
	\item[(i)] $\alpha^{(1)}_k=0$ whenever $|\mathcal O_{\alpha_k}|=1$,
	\item[(ii)] $\alpha_k^{(1)}$ is proportional to $\alpha_{\ell}^{(1)}$ whenever $\alpha_k$ and $\alpha_\ell$ belong to the same $\theta$-orbit.
\end{itemize}
This leads to the vectors shown in Table  \ref{tablebasis}, which are easily seen to be linearly independent, and $\h_1$ is a Cartan subspace by dimensional reasons.
\end{proof}

\begin{Proposition}
\label{lem:familyroots}
A root $\alpha\in\Phi$ satisfies $\alpha^{(-1)}\neq 0$  if and only if 
its $\theta$-orbit $\mathcal O_\alpha$ has cardinality $|\mathcal O_\alpha|>1$, in turn, if and only if it appears in the following Table
 \ref{tablerestricted}.
\begin{table}[h]
\[
\begin{array}{|c|c|c|c} \hline
\g & m &\alpha \\ \hline\hline
 \mathfrak{so}_{2n+2}(\C)& 2 &\pm\epsilon_i\pm\epsilon_{n+1}\;\;\text{where}\;\;i\neq n+1  \\ \hline

 \mathfrak{sl}_{2n}(\C)& 2 &\;\;\;\;\;\epsilon_i-\epsilon_j\;\;\text{where}\;\;j\neq i,2n+1-i \\ \hline

  E_6& 2 &\begin{gathered} \epsilon_i-\epsilon_j\;\;\text{where}\;\; 1\leq i,j\leq 6, j\neq i, 7-i
	\\  
	\alpha:=\tfrac{1}{2}\big((\lambda_1\epsilon_1+\cdots+\lambda_6\epsilon_{6})\pm\sqrt2 \epsilon_7\big)\;\;\text{where\;all}\;\;
	\lambda_k=\pm 1, \sum\lambda_k=0, \theta\cdot\alpha\neq \alpha
	\end{gathered} \\ 
	 \hline
 \mathfrak{so}_{8}(\C)& 3& \text{All}\;\;\pm\epsilon_i\pm\epsilon_j\;\;\text{except}\;\;\pm(\epsilon_1+\epsilon_2), \pm(\epsilon_1+\epsilon_3),\pm(\epsilon_2-\epsilon_3)\\ \hline
 \end{array}
 \]
 \caption{}
 \label{tablerestricted}
 \end{table}
\vskip-0.5cm\par
Furthermore:
\begin{itemize}
	\item[(i)] If $\g= \mathfrak{so}_{2n+2}(\C)$, all the roots of $\g$ in Table 
 \ref{tablerestricted} give rise to the same hyperplane $\Pi_\sigma$;
	\item[(ii)] If $\g=\mathfrak{so}_8(\mathbb C)$, all the roots of $\g$ in Table 
 \ref{tablerestricted} give rise to the same hyperplane $\Pi_\sigma$;
	\item[(iii)] If $\g=E_6$, the second type of roots of $\g$ in Table 
 \ref{tablerestricted} gives rise to the same collection of hyperplanes $\Pi_\sigma$ as those arising from the first type of roots. 
\end{itemize}
\end{Proposition}

\begin{proof}

If $|\mathcal O_\alpha|=1$, then $\alpha^{(-1)}=0$. If $|\mathcal O_\alpha|>1$, we assume by contradiction that $\alpha^{(-1)}=0$.
Then $\alpha^{(+1)}=\overline{\alpha^{(-1)}}=0$, so that $\alpha=\alpha^{(0)}$ and $\alpha$ is invariant under $\theta$, which is not possible. Since the condition $\alpha^{(-1)}\neq 0$ reads $(\alpha,\mathfrak c)\neq 0$, it is straightforward to use Table  \ref{tablebasis} and the list of all roots in \cite[\S 6.7]{MR1104219} to get Table \ref{tablerestricted}. If $m=2$, then $(\alpha,\mathfrak c)=(\alpha^{(-1)},\mathfrak c)=(\alpha^{(1)},\mathfrak c)$, whence the line orthogonal to $\Pi_\sigma$ in $\mathfrak c$ is $\mathbb C\alpha^{(1)}$ and we check that 
the first and second type of roots for $E_6$ in Table \ref{tablerestricted} give rise to the same collection of hyperplanes $\Pi_\sigma$. 
Finally, if $r=1$, it is clear that all the roots give rise to the same (trivial) hyperplane $\Pi_\sigma$.
\end{proof}

\begin{Theorem}
Let $\{\g, \theta, m\}$ be a periodically graded Lie algebra arising from a diagram automorphism $\theta$ of a simple Lie algebra $\g$ as in Table  \ref{tablediagram}.
Let $\sigma\in\Sigma$ with associated hyperplane $\Pi_\sigma$ in $\mathfrak c$ and $\alpha\in\Phi$ such that
$\sigma=\alpha\circ\rho$.
 Thanks to Proposition \ref{lem:familyroots}, we may assume without loss of generality that $\alpha$ is as in Table  \ref{tablerestrictedchangename}.
  Then the lift of a reflection $s\in W$ with $\Pi_s=\Pi_\sigma$ is given by $\exp\big(\frac{\pi}{2} (e_{\alpha}^{(0)}+e_{-\alpha}^{(0)})\big)$.
\begin{table}[h]
\[
\begin{array}{|c|c|c|c} \hline
\g & m &\alpha \\ \hline\hline
 \mathfrak{so}_{2n+2}(\C)& 2 &\alpha_{n+1}=\epsilon_n+\epsilon_{n+1} \\ \hline

 \mathfrak{sl}_{2n}(\C)& 2 &\;\;\;\;\;\alpha_i+\cdots+\alpha_{j-1}=\epsilon_i-\epsilon_j\;\;\text{where}\;\;i<j, j\neq 2n+1-i \\ \hline

  E_6& 2 &\begin{gathered} \alpha_i+\cdots+\alpha_{j-1}=\epsilon_i-\epsilon_j\;\;\text{where}\;\; 1\leq i<j\leq 6, j\neq 7-i
	\end{gathered} \\ 
	 \hline
 \mathfrak{so}_{8}(\C)& 3& \alpha_1=\epsilon_1-\epsilon_2\\ \hline
 \end{array}
 \]
 \caption{}
 \label{tablerestrictedchangename}
 \end{table}

\end{Theorem}
\vskip0.2cm\par\noindent
\begin{proof} We will freely make use of the uniform commutation relations in \cite[Remark 7.9(b)]{MR1104219}.
We depart with the cases of rank $r=1$. 
If $\g=\mathfrak{so}_{2n+2}(\mathbb C)$, we have
\begin{equation*}
\begin{aligned}
[e_{\pm\alpha}^{(0)},\alpha^{(1)}]&=\mp(\alpha^{(1)},\alpha)e_{\pm\alpha}^{(1)}=\mp e_{\pm\alpha}^{(1)}\;,\\
[e_{\pm\alpha}^{(0)},e_{\pm\alpha}^{(1)}]&=0\;.
\end{aligned}
\end{equation*}
Setting $J:=\frac{1}{2} \big(e_{\alpha}^{(0)}+e_{-\alpha}^{(0)}\big)
$, we see $[J,\alpha^{(1)}]=-\frac{1}{2}\big(e_{\alpha}^{(1)}-e_{-\alpha}^{(1)}\big)$ and
\begin{equation*}
\begin{aligned}
[J,e_{\alpha}^{(1)}-e_{-\alpha}^{(1)}]&=\frac{1}{2}[e_{\alpha}^{(0)}+e_{-\alpha}^{(0)},e_{\alpha}^{(1)}-e_{-\alpha}^{(1)}]=-\frac{1}{2}[e_{\alpha}^{(0)},e_{-\alpha}^{(1)}]
+\frac{1}{2}[e_{-\alpha}^{(0)},e_{\alpha}^{(1)}]\\
&=
2\alpha^{(1)}\;,
\end{aligned}
\end{equation*}
so that $J$ acts as a complex structure on the plane $\C\alpha^{(1)}\oplus \mathbb C(e_{\alpha}^{(1)}-e_{-\alpha}^{(1)})$. The claim follows.
If $\g=\mathfrak{so}_{8}(\mathbb C)$, the argument is very close and we only record the main steps: we have
\begin{align*}
[e_{\pm\alpha}^{(0)},3\alpha^{(1)}]&=
\mp3(\alpha^{(1)},\alpha)e_{\pm\alpha}^{(1)}=\mp 2e_{\pm\alpha}^{(1)}\\
[J,e_{\alpha}^{(1)}-e_{-\alpha}^{(1)}]&=-\frac{1}{2}[e_{\alpha}^{(0)},e_{-\alpha}^{(1)}]+\frac{1}{2}[e_{-\alpha}^{(0)},e_{\alpha}^{(1)}]=3\alpha^{(1)}
\end{align*}
and the claim follows again.

\vskip0.2cm\par
We now turn to $\g=\mathfrak{sl}_{2n}(\mathbb C)$.
Here $\alpha=\epsilon_i-\epsilon_j$ where $i<j$, $j\neq 2n+1-i$, and it is also convenient to introduce $\widetilde\alpha=\epsilon_i-\epsilon_{2n+1-j}$.
We consider the reduction process to rank $r=1$.
 A root $\beta=\epsilon_k-\epsilon_\ell\in\Phi$ belongs to $\Phi_x$ as in \eqref{phix} if and only if $(\beta,\Pi_\sigma)=(\beta^{(1)},\Pi_\sigma)=0$, which in turn holds if
\vskip0.2cm\par\noindent
\begin{itemize}
	\item[(i)] $\beta$ is invariant under $\theta\Longrightarrow\beta=\epsilon_k-\epsilon_{2n+1-k}$,
	\item[(ii)] $\beta^{(1)}$ is non-zero and proportional to $\alpha^{(1)}\Longrightarrow k,\ell\in\left\{i,j,2n+1-i,2n+1-j\right\}$ with $\ell\neq k,2n+1-k$. In other words $\beta\in\left\{\pm\alpha,\pm\theta\cdot\alpha,\pm\widetilde\alpha,\pm\theta\cdot\widetilde\alpha\right\}$. 
\end{itemize}  
\vskip0.2cm\par\noindent
In summary 
\vskip0.2cm\par\noindent
\begin{equation}
\label{eq:reductionhere}
\g^{\Pi_\sigma}=\mathfrak h\oplus\bigoplus_{\tiny{\begin{gathered} k\notin I\end{gathered}}}\g_{\epsilon_k-\epsilon_{2n+1-k}}\bigoplus_{\tiny{\begin{gathered} k,\ell\in I\end{gathered}}}\g_{\epsilon_k-\epsilon_\ell}\;,\qquad I:=\left\{i,j,2n+1-i,2n+1-j\right\}\;,
\end{equation}
and $\mathfrak m=[\g^{\Pi_\sigma},\g^{\Pi_\sigma}]\cong(n-2)\mathfrak{sl}_2(\mathbb C)\oplus\mathfrak{sl}_4(\mathbb C)$. We thus consider the
	periodically graded Lie algebra $\{\mathfrak{sl}_4(\mathbb C),\theta,2\}$ with Cartan subspace
$\mathbb C\alpha^{(1)}$, which is generated by the root spaces corresponding to the last contribution in \eqref{eq:reductionhere}. Since $\{\mathfrak{sl}_4(\mathbb C),\theta,2\}$ is isomorphic to $\{\mathfrak{so}_{2n+2}(\mathbb C),\theta,2\}$ for $n=2$, the result follows in a completely analogous fashion.
\vskip0.2cm\par
The strategy for $\g=E_6$ is similar. Here $\alpha=\epsilon_i-\epsilon_j$, $1\leq i,j\leq 6$, $j\neq i, 7-i$, and arguing as for $\mathfrak{sl}_{2n}(\mathbb C)$, we see that
\vskip0.3cm\par\noindent
$$
\g^{\Pi_\sigma}=\mathfrak h\oplus\bigoplus_{\tiny{\begin{gathered}\beta\in\Phi \\ \theta\cdot\beta=\beta\vphantom{\beta^{(1)}}\end{gathered}}} \g_\beta\bigoplus_{\tiny{\begin{gathered}\beta\in\Phi \\ {\beta^{(1)}\in\mathbb C^\times\alpha^{(1)}}\end{gathered}}}\!\!\!\!\!\g_\beta\;,
$$
\vskip0.3cm\par\noindent
which is a $46$-dimensional Levi subalgebra of $E_6$ (we won't write down the explicit expressions of the roots, since it is not particularly enlightening). It is necessarily isomorphic to the conformal algebra $\mathfrak{co}_{10}(\mathbb C)$.
In any case, we recognize the subalgebra of $\g^{\Pi_\sigma}$ generated by the root spaces $\g_{\epsilon_k-\epsilon_\ell}$  with $k,\ell\in I:=\left\{i,j,7-i,7-j\right\}$: it is a subalgebra isomorphic to $\mathfrak{sl}_4(\mathbb C)$, $\theta$-stable, and it includes $\mathbb C\alpha^{(1)}$. The result follows then again.
\end{proof}

\section*{Acknowledgments}
The authors would like to thank Giovanna Carnovale, Willem de Graaf, Francesco Esposito, and Oksana Yakimova for helpful and stimulating discussions.
The authors also thank the anonymous referee for carefully reading the manuscript and providing helpful comments that improved the paper.
The first author's research work was funded by FSU Jena.
The second author 
acknowledges the MIUR Excellence Department Project MatMod@TOV
awarded to the Department of Mathematics, University of Rome Tor Vergata, CUP E83C23000330006.  This article/publication was also supported by the ``National Group for
Algebraic and Geometric Structures, and their Applications'' GNSAGA-INdAM (Italy) and it is
based upon work from COST Action CaLISTA CA21109 supported by COST
(European Cooperation in Science and Technology), 
{\scriptsize\url{https://www.cost.eu}}.

\bibliographystyle{abbrv}
\bibliography{biblio}
\end{document}